\documentclass[11pt]{article}

\usepackage[numbers]{natbib}
\usepackage{amssymb, amsmath, url, amsthm, amssymb,graphicx}
\usepackage[T1]{fontenc}
\usepackage[latin1]{inputenc}

\newtheorem{Theorem}{Theorem}[section]
\newtheorem{Lemma}{Lemma}[section]

\newtheorem{Proposition}{Proposition}[section]

\newtheorem{Definition}{Definition}[section]
\newtheorem{Assumption}{Assumption}[section]

\begin{document}

\begin{center}
\noindent {\bf\Large A Functional Version of the ARCH Model}
\\[3ex]
{\large Siegfried H\"ormann$^{1,2}$\footnotetext[1]{
Corresponding author. E-mail address: \url{shormann@ulb.ac.be}}
\footnotetext[2]{
Research partially supported by the Banque National de Belgique and
Communaut\'e fran\c{c}aise de Belgique - Actions de Recherche Concert\'ees (2010--2015).}
, Lajos Horv\'ath$^3$, and Ron Reeder$^3$\footnotetext[3]{Research partially
supported by NSF grant DMS 0905400.}
}\bigskip

{\small
$^1$ D\'epartment de Math\'ematique,
Universit\'e Libre de Bruxelles,
CP 215,
Boulevard du Triomphe,
B-1050 Bruxelles,
Belgium\\
$^3$ Department of Mathematics, University of Utah, Salt Lake City, UT-84112-0090, USA}

\end{center}

\begin{small}
\noindent{\bf Abstract} \\[1ex]
\setlength{\baselineskip}{1.5em} Improvements in data acquisition
and processing techniques have lead to an almost continuous flow of information for financial
data. High resolution tick data are available and can be quite conveniently described by a continuous time process. It is therefore natural to ask for
possible extensions of financial time series models to a functional setup.
In this paper we propose a functional version of the popular
ARCH model.
We will establish conditions for the existence of
a strictly stationary solution, derive weak dependence and moment conditions,
show consistency of the estimators
and perform a small empirical study demonstrating how our model matches with
real data.
\\

\end{small}

\setlength{\parskip}{1.5ex plus0.5ex minus 0.5ex}
\setlength{\baselineskip}{1.4em}


\section{Introduction}\label{s:intro}

To date not many functional time series models exist to describe sequences of dependent observations.  Arguably the most popular
is the ARH(1), the autoregressive Hilbertian process of order $1$.
It is a natural extension of the scalar and vector valued
AR(1) process (cf.\ Brockwell and Davis~\cite{brockwell:davis:1991}).
Due to the fact that the ARH(1) model is mathematically and statistically quite
flexible and well established, it is used in practice for
modeling and prediction of continuous-time random experiments. We refer to
Bosq~\cite{bosq:2000} for a detailed treatment of moving averages, autoregressive
and general linear time series sequences.
Despite the prominent
presence in time series analysis it is clear that the applicability of
moving average and autoregressive processes is limited. To describe nonlinear models in the scalar and vector cases,  a
number of different approaches  have been introduced in the last decades. One of the most popular ones in econometrics is  the ARCH model of Engle~\cite{engle:1982}
and the more general GARCH model of Bollerslev~\cite{bollerslev:1990} which
have had an enormous impact on the modeling of financial
data. For surveys on volatility models we refer to Silvennoinen and Ter\"asvirta ~\cite{silv:2009}.
GARCH-type models are designed for the analysis of daily, weekly
or more general long-term period returns. Improvements
in data acquisition and processing techniques have lead to an almost
continuous flow of information for financial data
with online investment decisions.
High resolution tick data are available and can be quite conveniently described as functions.
It is therefore natural to ask for possible extensions of
these financial time series models to a functional setup. The idea is that  instead of
a scalar return sequence $\{y_k,\, 1\leq k\leq T\}$ we have a functional time series
$\{y_k(t),\, 1\leq k\leq T,\, 0\leq t\leq S\}$, where $y_k(t)$ are intraday (log-)returns on day $k$ at time $t$. In other words if $\{P_k(t),\, 1\leq k\leq T,\, 0\leq t\leq S\}$ is the underlying price process, then $y_k(t)=\log P_k(t)-\log P_k(t-h)$ for the
desired time lag $h$, where we will typically set $h= 5 \mathrm{min}$.
By rescaling we can always assume that $S=1$ and then the interval $[0,1]$ represents one trading day.

We notice that a daily
segmentation of the data is natural and preferable to only one continuous
time process $\{y(s),\, 0\leq s\leq T\}$, say, for all $T$ days of our
sample (cf.\ Harrison et al.~\cite{harrison:1984}, Barndorff-Nielsen and Shephard~\cite{barndorff:2004}, Zhang et al.~\cite{Zhang:2005}, Barndorff-Nielsen et al.~\cite{barndorff:2008}, and Jacod et al.~\cite{jacod:2009}). Due to the time laps between trading days (implying e.g.\ that
opening and  closing prices do not necessarily coincide) one continuous time
model might not be suitable for a longer period. Intraday volatilities of the euro-dollar rates investigated by Cyree et al.~\cite{cyree:2005} empirically can be considered as daily curves. Similarly, Gau~\cite{gau:2005}  studied the shape of the intraday volatility curves of the Taipei FX market. Angelidis and Degiannakis~\cite{angelidis:2008}  compared predictions based on intra-day and inter-day data. Elezovi\'c~\cite{elezovic:2009} modeled bid and ask prices as continuous functions. The spot exchange rates in Fatum and Pedersen~\cite{fatum:2009} can be considered as functional observations as well. Evans and Speight~\cite{evans:2010} uses 5-min returns for Euro-Dollar, Euro-Sterling and Euro-Yen exchange rates.

In this paper we propose a functional ARCH model. 
Usually time series
are defined  by stochastic recurrence equations establishing the
relationship between past and future observations.
The question preceding any further analysis is whether such an equation has
a (stationary) solution. For the scalar ARCH necessary and sufficient  conditions have been derived by Nelson~\cite{nelson:1990}. Interestingly, these results cannot be transferred directly to multivariate extensions (cf.\ Silvennoinen and Ter\"asvirta ~\cite{silv:2009}).
Due to the complicated dynamics of
multivariate ARCH/GARCH type models (MGARCH), finding the necessary and sufficient
conditions for the existence of stationary solutions to the defining equations is a difficult problem. Also the characterization of the existence of the moments in GARCH$(p,q)$ equations is given by very involved formulas (cf.\ Ling, S. and  McAleer~\cite{ling:2002}).
It is therefore not
surprising that in a functional setup, i.e.\ when dealing with
intrinsically infinite dimensional objects,
some  balancing between generality and mathematical
feasibility of the model is required.

In Section~\ref{s:farch} we propose
a model for which we provide conditions for the existence of a unique stationary solution.  These conditions are not too far from being optimal. We will also
study the dependence structure of the model, which is useful in many
applications, e.g.\ in estimation which will be treated in Section~\ref{s:estimation}.
 We also provide an example illustrating that the proposed functional ARCH
 model is able to capture typical characteristics
of high
frequency returns, see Section~\ref{s:empirical}.

In this paper we  use  the following notation.
Let  $F$ denote  a generic function space. Throughout
this consists of
real valued functions with domain $[0,1]$.
In many applications $F$ will be equal to $\mathcal{H}$=$L^2([0,1]),$
 the Hilbert space of square integrable functions with norm $\|x\|_\mathcal{H}=\big(\int_0^1x^2(s)ds\big)^{1/2}$
which is generated by the inner product $\langle x,y\rangle=\int_0^1x(s)y(s)ds$ for $x,y\in\mathcal{H}$. Another important example
is $F=\mathcal{C}[0,1]$. This is the space of continuous functions
on $[0,1]$ equipped with the sup-norm $\|x\|_\infty=\sup_{t\in[0,1]}|x(t)|.$
By $F^+$ we denote the set of non-negative functions in $F$.
To further lighten notation we
shall often write $x$ when we mean $\{x(t),\,t\in[0,1]\}$,
or $\beta$ for
integral kernels $\{\beta(t,s),\,0\leq t\leq 1,\, 0\leq s\leq 1\}$
as well as for the corresponding operators.
If $x,y\in F$ then $xy$ stands for pointwise multiplications, i.e.  $xy=\{x(s)y(s),\,s\in[0,1]\}$.
Since integrals will always be taken over the unit interval we shall henceforth
simply write $\int x(t)dt$.
A random function $X$ with values in $\mathcal{H}$ is said to be in $L_\mathcal{H}^p$ if
$\nu_p(X)=\big(E\|X\|_\mathcal{H}^p\big)^{1/p}<\infty$. 


\section{The functional ARCH model}\label{s:farch}

We start with the following general definition.
\begin{Definition}\label{d:S-ARCH}
Let
$\{\varepsilon_k\}$ be a sequence of independent and identically
distributed random functions in $F$.  Further let
$\beta:F^+\to F^+$ be a non-negative operator and let $\delta\in F^+$.
Then an $F$-valued process $\{y_k(s),\,k\in\mathbb{Z},\,s\in[0,1]\}$
is called a functional ARCH(1) process in $F$ if the following holds:
\begin{equation}\label{e:y}
y_k=\varepsilon_k\sigma_k
\end{equation}
and
\begin{equation}\label{e:sigma}
\sigma_k^2=\delta+\beta(y_{k-1}^2).
\end{equation}
\end{Definition}
The assumption for the existence of processes satisfying (\ref{e:y}) and
(\ref{e:sigma}) depends on the choice of $F$.
So next we  specify $F$ and put some restrictions
on the operator $\beta$.
Our first result gives a sufficient condition for the existence of a strictly stationary
solution when $F=\mathcal{H}$.
We will assume that $\beta$ is
a (bounded) kernel operator defined by
\begin{equation}\label{e:betakernel}
\beta(x)(t)=\int\beta(t,s)x(s)ds,\quad x\in \mathcal{H}.
\end{equation}
Boundedness is e.g.\ guaranteed by finiteness of the Hilbert-Schmidt norm:
\begin{equation}\label{e:bounded_beta}
\|\beta\|_\mathcal{S}^2=\int\int\beta^2(t,s)dsdt<\infty.
\end{equation}
\begin{Theorem}\label{th:Hexist}
Let $\{y_k\}$ be the process given in Definition~\ref{d:S-ARCH}  with $F=\mathcal{H}$ and
$\beta$ given in \eqref{e:betakernel}, such that the operator $\beta$ is bounded.
Define $K(\varepsilon_1^2)=\big(\int\int\beta^2(t,s)\varepsilon_1^4(s)dsdt\big)^{1/2}$.
If there is some $\alpha>0$ such that
$E\big\{K(\varepsilon_1^2)\big\}^\alpha<1$, then \eqref{e:y} and \eqref{e:sigma} have
a unique strictly stationary solution in $\mathcal{H}$. Furthermore, $\sigma_k^2$ is of the form
\begin{equation}\label{e:bernoullishift}
\sigma_k^2=g(\varepsilon_{k-1},\varepsilon_{k-2},\ldots),
\end{equation}
with some measurable functional $g:\mathcal{H}^{\mathbb{N}}\to \mathcal{H}$.
\end{Theorem}

It follows that $\{\sigma_k\}$ and $\{y_k\}$ are not just strictly stationary
but also ergodic (cf.\ Stout~\cite{stout:1974}).
Let $\mathcal{F}_k$ be the $\sigma$-algebra generated by the sequence
$\{\varepsilon_i,\, i\leq k\}$. If \eqref{e:y} and \eqref{e:sigma}
have a stationary solution and if we assume that $E\varepsilon_k(t)=0$,
$E\varepsilon^2(t)<\infty$ and $E\sigma_k^2(t)<\infty$ for all $t\in[0,1]$,
then due to \eqref{e:bernoullishift} it is easy to see that
$$
\mathrm{corr}\big\{y_k(t),y_{k}(s)|\mathcal{F}_{k-1}\big\}=
\mathrm{corr}\big\{\varepsilon_k(t),\varepsilon_{k}(s)\big\}.
$$
Since by our assumption $\{\varepsilon_k,\, k\in\mathbb{Z}\}$
is stationary, the conditional correlation is independent of $k$
and can be fully described by the covariance kernel
$C_\varepsilon(t,s)=\mathrm{Cov}(\varepsilon(t),\varepsilon(s))$.
However, we have
$
\mathrm{Cov}\big\{y_k(t),y_{k}(s)|\mathcal{F}_{k-1}\big\}=\sigma_k(t)\sigma_k(s) C_{\varepsilon}(s,t).
$
This is in accordance with the {\em constant conditional correlation} (CCC)
multivariate GARCH models of Bollerslev~\cite{bollerslev:1990} and
Jeantheau~\cite{jeantheau:1998}.

Our next result shows that $\sigma_k^2$ of \eqref{e:bernoullishift}
can be geometrically approximated with $m$-dependent variables, which
establishes weak dependence of the processes \eqref{e:y} and
\eqref{e:sigma}.

\begin{Theorem}\label{th:Hweak}
Assume that the conditions of Theorem~\ref{th:Hexist} hold.
Let $\{\varepsilon_k'\}$ be an independent copy of
$\{\varepsilon_k\}$ and define
$
\sigma_{km}^2=g(\varepsilon_{k-1},\varepsilon_{k-2},\ldots,\varepsilon_{k-m},\varepsilon_{k-m-1}',
\varepsilon_{k-m-2}',\ldots).
$
Then
\begin{equation}\label{e:Hweak}
E\{\|\sigma_k^2-\sigma^2_{km}\|_\mathcal{H}\}^\alpha \leq cr^m,
\end{equation}
with some $0<r=r(\alpha)<1$ and $c=c(\alpha)<\infty$.
\end{Theorem}

To better understand the idea behind our result we remark the following. Assume that we redefine
$$
\sigma_{km}^2=g(\varepsilon_{k-1},\varepsilon_{k-2},\ldots,\varepsilon_{k-m},\varepsilon_{k-m-1,k-m}^{(k)},\varepsilon_{k-m-2,k-m}^{(k)},\ldots),
$$
where $\{\varepsilon_{\ell,i}^{(k)},\ell,i,k \in \mathbb{Z}\}$ are independent copies of $\{\varepsilon_\ell,\ell\in\mathbb{Z}\}$.
In other words,
every $\sigma_k^2$  gets its "individual" copy of $\{\varepsilon_{\ell,i}^{(k)}\}$
to define the approximations.
It can be easily seen that then for any fixed $m\geq 1$, $\{\sigma_{km}^2,\,k\in\mathbb{Z}\}$
form $m$-dependent sequences, while the value on the left hand side in inequality \eqref{e:Hweak}
doesn't change. As we have shown in our recent papers
\cite{aue:hormann:horvath:huskova:steinbach:2009+}
and \cite{hoermann:kokoszka:2010}, approximations like (\ref{e:Hweak}) are particularly useful in
studying large sample properties of functional data. We use (\ref{e:Hweak}) to provide conditions for the existence of moments of the stationary solutions of \eqref{e:y} and \eqref{e:sigma}. It also follows immediately from (\ref{e:Hweak}), that if \eqref{e:y} and \eqref{e:sigma} are solved starting with some initial values $y_0^*$ and $\sigma_0^*$, then the
effect of the initial values dies out exponentially fast.

In a finite dimensional vector space all norms are equivalent. This
is no longer true in the functional (infinite dimensional) setup and
whether a solution of \eqref{e:y} and \eqref{e:sigma} exists
depends on the choice of space and norm of the state space. Depending
on the application, it might be more convenient to work in a different space.
We give here the analogue of Theorems~\ref{th:Hexist} and \ref{th:Hweak} for
a functional ARCH process in $\mathcal{C}[0,1]$.

\begin{Theorem}\label{th:Cexist}
Let $\{y_k\}$ be the process given in Definition~\ref{d:S-ARCH} with $F=\mathcal{C}[0,1]$ and
define $H(\varepsilon_1^2)=\sup_{0\leq t\leq 1}\int\beta(t,s)\varepsilon_1^2(s)ds$.
If there is some $\alpha>0$ such that
$E\big\{H(\varepsilon_1^2)\big\}^\alpha<1$, then \eqref{e:y} and \eqref{e:sigma} have
a unique strictly stationary solution in $\mathcal{C}[0,1]$. Furthermore, $\sigma_k^2$ can be
represented as in \eqref{e:bernoullishift}.
In addition the proposition of Theorem~\ref{th:Hweak} holds, with
\eqref{e:Hweak} replaced by
$$
E\big\{\|\sigma_k^2-\sigma_{km}^2\|_\infty\big\}^\alpha\leq cr^{m}.
$$
\end{Theorem}

We continue with some immediate consequences of our theorems. We start with
conditions for the existence of the moments of the stationary solution of
  \eqref{e:y} and \eqref{e:sigma}.
\begin{Proposition}\label{p:momentsH}
Assume that the conditions of Theorem~\ref{th:Hexist} hold. Then
\begin{equation}\label{e:momH1}
E\big\{\|\sigma_0^2\|_\mathcal{H}\big\}^{\alpha}<\infty.
\end{equation}
If
\begin{equation}\label{e:momH2}
E\big\{\|\sigma_0\|_\mathcal{H}\big\}^{\alpha}<\infty
\end{equation}
and
\begin{equation}
E\big\{\|\varepsilon_0\|_\infty\big\}^{\alpha}<\infty,
\end{equation}
then
\begin{equation}\label{e:momH3}
E\big\{\|y_0\|_\mathcal{H}\big\}^{\alpha}<\infty.
\end{equation}
\end{Proposition}

\begin{Proposition}\label{p:momentsC}
Assume that the conditions of Theorem~\ref{th:Cexist} hold.  Then the analogue of Proposition~\ref{p:momentsH} holds, with
$\|\cdot\|_\mathcal{H}$ in \eqref{e:momH1}--\eqref{e:momH3} replaced
by $\|\cdot\|_\infty$.
\end{Proposition}We would like to point out that it is not assumed that the innovations $\varepsilon_k$ have finite variance. We only need that $\varepsilon_k$ have some moment of order  $\alpha>0$, where $\alpha>0$ can be as small as we wish.
Hence our model allows for innovations as well as observations with heavy tails.\\

According to Propositions~\ref{p:momentsH} and \ref{p:momentsC}, if the innovation $\varepsilon_0$ has enough
moments, then so does $\sigma_0^2$ and $y_0$. The next result shows a connection between
the moduli of continuity of $\varepsilon_0$ and $y_0$. Let
$$
\omega(x,h)=\sup_{0\leq t \leq 1-h}\sup_{0\leq s\leq h}|x(t+s)-x(t)|
$$
denote the modulus of continuity of a function $x(t)$.

\begin{Proposition}\label{p:continuity}
 We assume that the conditions of
Theorem~\ref{th:Cexist} are satisfied with $\alpha=p>0$. If
$
E\big\{\|\varepsilon_0\|_\infty\big\}^p<\infty
$ and
$
\lim_{h\to 0}E\big\{\omega(\varepsilon_0,h)\big\}^p = 0,
$ then
$
\lim_{h\to 0}E\big\{\omega(y_0,h)\big\}^p = 0.
$
\end{Proposition}

According to Theorems~\ref{th:Hweak} and \ref{th:Cexist}, the stationary
solution of \eqref{e:y} and \eqref{e:sigma} can be approximated
with stationary, weakly dependent sequences with values in $\mathcal{H}$ and in
$\mathcal{C}[0,1]$, respectively. We provide two further results which establish the weak dependence
structure of $\{y_k\}$.

\begin{Proposition}\label{p:wdH}
We assume that the conditions of Theorem~\ref{th:Hexist} are satisfied
with $\alpha=\frac{p}{2}$ and
\begin{equation}\label{e:finmome}
E\big\{\|\varepsilon_0\|_\infty\big\}^p<\infty.
\end{equation}
 Then
\begin{equation}\label{e:wdH}
E\big\{\|y_k-y_{km}\|_\mathcal{H}\big\}^p\leq c\gamma^m,\quad -\infty<k<\infty, m\geq 1,
\end{equation}
with some $0<c<\infty$ and $0<\gamma<1$, where $y_{km}=\varepsilon_k\sigma_{km}$.
\end{Proposition}

It follows from the definitions that the distribution of the $y_k-y_{km}$
does not depend on $k$. Hence the expected value in \eqref{e:wdH} does not depend
on $k$. A similar result holds in $F=\mathcal{C}[0,1]$ under the sup-norm.
\begin{Proposition}\label{p:wdC}
We assume that the conditions of Theorem~\ref{th:Cexist} are satisfied
with $\alpha=\frac{p}{2}$ and that \eqref{e:finmome} holds. Then
\begin{equation}\label{e:wdC}
\sup_{0\leq t\leq 1}E\big |y_k-y_{km}|^p\leq c\gamma^m,\quad -\infty<k<\infty, m\geq 1,
\end{equation}
with some $0<c<\infty$ and $0<\gamma<1$, where $y_{km}=\varepsilon_k\sigma_{km}$.
\end{Proposition}
As in case of Proposition~\ref{p:wdH}, the expected value in \eqref{e:wdC} does not
depend on $k$.

\section{Estimation}
\label{s:estimation}
In this section we propose estimators for the function $\delta$ and the operator $\beta$ in model
\eqref{e:y}--\eqref{e:sigma} which are not known in practice. The procedure is developed for the
important case where $F=\mathcal{H}$ and $\beta$ is given as in \eqref{e:betakernel}.
We show that our problem is related to the estimation of the autocorrelation operator in
the ARH(1) model which has been intensively studied in Bosq~\cite{bosq:2000}.
However, the theory developed in Bosq~\cite{bosq:2000} is not directly applicable
as it requires independent innovations in the ARH(1) process, whereas, as we will see
below, we can only assume weak white noise (in Hilbert space sense).

We will impose the following
\begin{Assumption}\label{a:estimation}
\begin{description}
\item[(a)] $E\varepsilon_0^2(t)=1$ for any $t\in[0,1]$.
\item[(b)] The assumptions of Theorem~\ref{th:Hweak} hold with $\alpha=2$.
\end{description}
\end{Assumption}
Assumption~\ref{a:estimation} (a) is needed to guarantee the identifiability of the model.
Part (b) of the assumption guarantees the existence of a stationary solution of the model
\eqref{e:y}--\eqref{e:sigma} with moments of order 4. It is necessary
to make the moment based estimator proposed below working.
An immediate consequence of Assumption~\ref{a:estimation}
is that \eqref{e:bounded_beta} holds, i.e.\ $\beta$ is a Hilbert Schmidt operator.

We let $m_2$ denote the mean function of the $y_k^2$ and introduce
$$
\nu_k = y_k^2-\sigma_k^2=\{(\varepsilon_k^2(s)-1)\sigma_k^2(s),\,s\in[0,1]\}.
$$
Then by adding $\nu_k$ on both sides of \eqref{e:sigma} we obtain
$$
y_k^2=\delta+\beta(y_{k-1}^2)+\nu_k.
$$
Since $\beta$ is a linear operator we obtain after subtracting $m_2$ on both sides of the
above equation
\begin{equation}\label{e:ar_rep}
y_k^2-m_2=\delta-m_2+\beta(m_2)+\beta(y_{k-1}^2-m_2)+\nu_k.
\end{equation}
It can be easily seen that under Assumption~\ref{a:estimation} $E\nu_k=0$ (where 0 stands
for the zero function). Notice also that the expectation commutes with bounded
operators, and hence
that $E(\beta(y_k^2-m_2))=\beta(E(y_k^2-m_2))=0$.
Consequently, taking expectations on both sides of \eqref{e:ar_rep} yields that
\begin{equation}\label{e:delta}
\delta-m_2+\beta(m_2)=0.
\end{equation}
Thus, \eqref{e:ar_rep} can be rewritten in the form
\begin{equation}\label{e:ar_rep1}
Z_k=\beta(Z_{k-1})+\nu_k\quad \text{with}\quad Z_k=y_k^2-m_2.
\end{equation}

Model \eqref{e:ar_rep1} is the autoregressive Hilbertian model of order 1, short ARH(1). For estimating the autocorrelation operator $\beta$ we may use the estimator proposed in Bosq~\cite[Chapter 8]{bosq:2000}.
We need to be aware, however, that the theory in \cite{bosq:2000}
has been developed for ARH processes with strong white noise innovations, i.e.\ independent innovations $\{\nu_k\}$. In our setup the
$\{\nu_k\}$ form only a {\em weak white noise} sequence, i.e.\ for any $n\neq m$ we have
$$
E\|\nu_n\|_\mathcal{H}^2<\infty\quad\text{and}\quad E\langle \nu_n,x\rangle\langle \nu_m,y\rangle=0\quad\forall x,y\in \mathcal{H},
$$
and the covariance operator of $\nu_n$ is independent of $n$.
Thus the theory in \cite{bosq:2000} cannot be directly applied.
We will study the estimation of $\beta$ in Section~\ref{ss:far}.

Once $\beta$ is estimated by some $\hat\beta$ say,
we obtain an estimator for $\delta$ via equation \eqref{e:delta}:
\begin{equation}\label{e:hat_delta}
\hat\delta=\hat{m}_2-\hat\beta(\hat{m}_2),
\end{equation}
where we use
\begin{equation}\label{e:estim_s}
\hat{m}_2=\frac{1}{N}\sum_{k=1}^N y_k^2.
\end{equation}

Let
$
\|\beta\|_\mathcal{L}=\sup_{x\in\mathcal{H}}\{\|\beta(x)\|_\mathcal{H}:\,\|x\|\leq 1\}
$
be the operator norm of $\beta$. Recall that $\|\beta\|_\mathcal{L}\leq \|\beta\|_\mathcal{S}$. The following Lemma shows that consistency of $\hat\beta$
implies consistency of $\hat\delta$.
\begin{Lemma}
Let Assumption~\ref{a:estimation} hold. Let $\hat\delta=\hat\delta_N$ be given as in \eqref{e:hat_delta}.
Then
$$
\|\hat\delta_N-\delta\|_\mathcal{H}=O_P(1)\times\left(N^{-1/2}+
\|\hat\beta_N-\beta\|_\mathcal{L}\right).
$$
\end{Lemma}
\begin{proof}
We have
\begin{align*}
\|\hat\delta_N-\delta\|_\mathcal{H}&\leq \|\hat{m}_2-m_2\|_\mathcal{H}+\|\hat\beta(\hat{m}_2)-\beta(\hat{m}_2)\|_\mathcal{H}+\|\beta(\hat{m}_2)-\beta(m_2)\|_\mathcal{H}\\
&\leq \|\hat{m}_2-m_2\|_\mathcal{H}+\|\hat\beta-\beta\|_\mathcal{L}\|\hat{m}_2\|_\mathcal{H}+\|\beta\|_\mathcal{L}\|\hat{m}_2-m_2\|_\mathcal{H}.
\end{align*}
The result follows once we can show that $\|\hat{m}_2-m_2\|_\mathcal{H}=\|\hat{m}_{2,N}-m_2\|_\mathcal{H}=O_P(N^{-1/2})$.
To this end we notice that by stationarity of $\{y_k^2\}$
\begin{align*}
 E\|\hat{m}_2-m_2\|_\mathcal{H}^2&=E\left\|
\frac{1}{N}\sum_{k=1}^N (y_k^2-m_2)
\right\|_\mathcal{H}^2\\
&=\frac{1}{N}\sum_{k=-(N-1)}^{N-1}\left(1-\frac{|k|}{N}\right)E\left\langle y_0^2-m_2,y_k^2-m_2\right\rangle\\
&\leq \frac{1}{N}\left(E\left\|y_0^2-m_2\right\|_\mathcal{H}+2\sum_{k=1}^{\infty}\left|E\left\langle y_0^2-m_2,y_k^2-m_2\right\rangle\right|\right).
\end{align*}
By construction $y_0^2$ and the approximation $y_{kk}^2$ are independent. Repeated application of the Cauchy-Schwarz inequality together with Assumption~\ref{a:estimation} (a) yield that
\begin{align*}
&\left|E\left\langle y_0^2-m_2,y_k^2-m_2\right\rangle\right|=\left|E\left\langle y_0^2-m_2,y_k^2-y_{kk}^2\right\rangle\right|\\
&\qquad  \leq \left(E\left\|y_0^2-m_2\right\|_\mathcal{H}^2\right)^{1/2}\left(E\left\|y_k^2-y_{kk}^2\right\|_\mathcal{H}^2\right)^{1/2}\\
&\qquad  = \left(E\left\|\sigma_0^2-m_2\right\|_\mathcal{H}^2\right)^{1/2}\left(E\left\|\sigma_k^2-\sigma_{kk}^2\right\|_\mathcal{H}^2\right)^{1/2}.
\end{align*}
Combining these estimates with Theorem~\ref{th:Hweak} shows that $E\left\|\hat{m}_2-m_2\right\|_\mathcal{H}=O(N^{-1/2})$.
\end{proof}

\subsection{Estimation of $\beta$}\label{ss:far}
We now turn to the estimation of the autoregressive operator $\beta$ in the
ARH(1) model \eqref{e:ar_rep1}.
It is instructive to focus first on the univariate case
$Z_n = \beta Z_{n-1} + \nu_n$, in which all quantities are scalars.
We assume $E\nu_n=0$ which implies $EZ_n=0$.
We also assume that $|\beta|<1$, so that there is
a stationary solution such that $\nu_n$ is uncorrelated with $Z_{n-1}$.
Then, multiplying the AR(1) equation by $Z_{n-1}$ and taking the expectation,
we obtain $\gamma_1 = \beta \gamma_0$, where
$\gamma_k = E[Z_n Z_{n+k}] = \mathrm{cov}(Z_n, Z_{n+k})$.
The autocovariances $\gamma_k$ are estimated in the usual way by the sample
autocovariances
$$
\hat \gamma_k = \frac{1}{N} \sum_{j=1}^{N-k} Z_j Z_{j+k},
$$
so the usual estimator of $\beta$ is $\hat \beta = \hat \gamma_1/\hat \gamma_0$.
This is the so-called {\em Yule-Walker estimator} which is optimal in many ways,
see Chapter 8 of Brockwell and Davis~\cite{brockwell:davis:1991}.

In the functional setup we will replace condition $|\beta|<1$ with $\|\beta\|_\mathcal{S}<1$. Notice that this
condition is guaranteed by Assumption~\ref{a:estimation} and that it will imply the existence of a
weakly stationary solution of \eqref{e:ar_rep1} of the form
$$
Z_n=\sum_{j\geq 0}\beta^j(\nu_{n-j}),
$$
where $\beta^j$ is the $j$-times iteration of the operator $\beta$ and $\beta^0$ is
the identity mapping. The estimator
for the operator $\beta$ obtained in \cite{bosq:2000}
is formally analogue to the scalar case. We need instead of  $\gamma_0$ and $\gamma_1$
the covariance operator
$$
C_0(\cdot)=E\left[
\langle Z_1,\cdot\rangle Z_1
\right]
$$
and the cross-covariance operator
$$
C_1(\cdot)=E\left[
\langle Z_1,\cdot\rangle Z_2
\right].
$$
One can show by similar arguments as in the scalar case that
$$
\beta=C_1C^{-1}.
$$
To get an explicit form
let $\lambda_1\geq \lambda_2\geq \cdots$
be the eigenvalues of $C$ and let $e_1,e_2,\ldots$ be the corresponding
eigenfunctions, i.e.\ $C(e_i)=\lambda_ie_i$. We assume that
$e_j$ are normalized to satisfy $\|e_j\|_\mathcal{H}=1$. Then $\{e_j\}$
forms an orthonormal basis (ONB) of $\mathcal{H}$ and we obtain the
following spectral decomposition of the operator $C$:
\begin{equation}\label{e:spectral}
C(y)=\sum_{j\geq 1}\lambda_j\langle e_j,y\rangle e_j.
\end{equation}
From \eqref{e:spectral} we get formally that
\begin{equation}\label{e:cinv}
C^{-1}(y)=\sum_{j\geq 1}\lambda_j^{-1}\langle e_j,y\rangle e_j,
\end{equation}
and hence
\begin{align}
\beta(y)&=C_1C^{-1}(y)=E\left(\left\langle Z_1,\sum_{j\geq 1}\lambda_j^{-1}\langle e_j,y\rangle e_j
\right\rangle Z_2\right)\nonumber\\
&=\sum_{j\geq 1}\lambda_j^{-1}\langle e_j,y\rangle
E\left(\langle Z_1,e_j\rangle Z_2\right).\label{e:rep_beta_op}
\end{align}
Using $Z_2=\sum_{i\geq 1}\langle Z_2,e_i\rangle e_i$ we
obtain that the corresponding kernel is
\begin{equation}\label{e:rep_beta}
\beta(t,s)=\sum_{i,j\geq 1}\lambda_j^{-1}E\left(\langle Z_1,e_j\rangle\langle Z_{2},e_i \rangle \right) e_j(s)e_i(t).
\end{equation}

If $\lambda_j=0$ for all $j > p\geq 1$, then the covariance operator is finite rank and
we can replace \eqref{e:spectral} and \eqref{e:cinv} by finite expansions with the
sum going from $1$ to $p$. In this case, all our mathematical operations so far are
well justified. However, when all $\lambda_j>0$ then we need to be aware that $C^{-1}$ is not
bounded on $\mathcal{H}$. To see this note that
$\lambda_j\to 0$ if $j\to \infty$ (this follows from the fact that $C$ is a Hilbert-Schmidt operator).
Consequently,  $\|C^{-1}(e_j)\|_\mathcal{H}=\lambda_j^{-1}\to\infty$
for $j\to\infty$.
It can be easily seen that this operator is bounded only on
$$
D=\left\{y\in\mathcal{H}:\,\sum_{j\geq 1}\frac{\langle e_j,y\rangle^2}{\lambda_j^2}<\infty
\right\}.
$$
Nevertheless, we can show that the representation \eqref{e:rep_beta_op} holds
for all $y\in\mathcal{H}$ by using a direct expansion of $\beta(t,s)$.
Since the eigenfunctions $\{e_k,\, k\geq 1\}$ of $C$ form an ONB
of $\mathcal{H}$ it follows that $\{e_k\otimes e_\ell,\, k,\ell\geq 1\}$
($e_k\otimes e_\ell=\{e_k(s)e_\ell(t),\, (s,t)\in[0,1]^2\}$) forms an ONB
of $L^2([0,1]^2)=\mathcal{H}\otimes \mathcal{H}$. This is again a Hilbert space with
inner product
$$
\langle x,y\rangle_{\mathcal{H}\otimes \mathcal{H}}=\int\int x(t,s)y(t,s)dtds.
$$
Note that $\|\beta\|_{\mathcal{H}\otimes \mathcal{H}}=\|\beta\|_\mathcal{S}<\infty$
and hence the kernel function $\beta\in\mathcal{H}\otimes \mathcal{H}$. (Be aware, that for the sake of a lighter
notation we don't distinguish between kernel and operator $\beta$.) Consequently $\beta(t,s)$ has the
representation
$$
\beta=\sum_{k,\ell\geq 1}\beta_{k,\ell}e_k\otimes e_\ell.
$$
As we can write
$$
Z_{n+1}=\sum_{k,\ell\geq 1}\beta_{k,\ell}\langle Z_n,e_k\rangle e_\ell+v_{n+1}
$$
it follows that
$$
\langle Z_{n+1},e_i \rangle \langle Z_n,e_j\rangle=
\sum_{k\geq 1}\beta_{k,i}
\langle Z_{n},e_k\rangle \langle Z_n,e_j\rangle
+\langle \nu_{n+1},e_i \rangle \langle Z_n,e_j\rangle
$$
and by taking expectations on both sides of the above equation that
$$
E\langle Z_{2},e_i \rangle \langle Z_1,e_j\rangle=
\sum_{k\geq 1}\beta_{k,i}
\langle C(e_k),e_j\rangle=\beta_{j,i}\lambda_j.
$$
Here we used the fact that $\{\nu_k\}$ is weak white noise. It implies that
$E\langle B(v_k),x\rangle\langle v_\ell,y\rangle$ is zero for any bounded
operator $B$ and all $x,y\in \mathcal{H}$ and all $k\neq \ell$. Hence the expansion
$Z_k=\sum_{j\geq 0}\beta^j(\nu_{k-j})$ provides $E\langle \nu_{n+1},e_i \rangle \langle Z_n,e_j\rangle=0$. This shows again \eqref{e:rep_beta}.

We would like to obtain now an estimator for $\beta$ by using a finite
sample version of the above relations.
To this end we set
$$
\hat{C}(y)=\frac{1}{N}\sum_{k=1}^N
\langle Z_k,y\rangle Z_k
\quad\text{and}\quad
\hat{C}_1(y)=\frac{1}{N}\sum_{k=1}^{N-1}
\langle Z_k,y\rangle Z_{k+1},\quad y\in \mathcal{H}.
$$
The estimator in \cite{bosq:2000} and the estimator we also propose here
is of the form
$$
\hat{\beta}(y;K)=\pi_K\hat{C}_1\widehat{C^{-1}}(y;K),
$$
where
\begin{equation}\label{e:cinv_s}
\widehat{C^{-1}}(y;K)=\sum_{j=1}^K\hat\lambda_j^{-1}\langle \hat{e}_j, y\rangle \hat{e}_j,
\end{equation}
$(\hat\lambda_j,\hat{e}_j)$ are the eigenvalues (in descending order) and
the corresponding eigenfunctions
of $\hat C$ and $p_K$ is the orthogonal projection onto the subspace \ $\mathrm{span}(\hat e_1,
\ldots, \hat e_K)$.
We notice that this estimator is not depending on the sign
of the $\hat e_j$'s.
The corresponding kernel is given as
\begin{equation}\label{e:betahat}
\hat{\beta}(t,s;K)=\frac{1}{N-1}\sum_{k=1}^{N-1}\sum_{j=1}^K\sum_{i=1}^K
\hat\lambda_j^{-1}\langle Z_k,\hat{e}_j\rangle \langle Z_{k+1},\hat{e}_i\rangle
\hat{e}_j(s)\hat{e}_i(t),
\end{equation}
and the signs of the $\hat e_j$ cancel out.
In practice eigenvalues and eigenfunctions
of an empirical covariance operator can be conveniently computed
with the package {\tt fda} for the statistical software {\tt R}.
The estimator \eqref{e:betahat} is the empirical version of the finite
expansion
$$
\beta(t,s;K)=\sum_{i=1}^K\sum_{j=1}^K\lambda_j^{-1}E\left(\langle Z_1,e_j\rangle\langle Z_{2},e_i \rangle \right) e_j(s)e_i(t)
$$
of \eqref{e:rep_beta}.

If the innovations $\{\nu_k\}$ are i.i.d.\ Bosq~\cite{bosq:2000} proves under some technical conditions
consistency of the estimator \eqref{e:betahat} when $K=K(N)$:
$$
\|\beta-\hat\beta(K(N))\|_\mathcal{L}=o_P(1)\quad \text{as $N\to\infty$.}
$$
The choice of
$K(N)$ depends on the decay rate  of the eigenvalues, which is not known in
practice. Empirical results (see Didericksonet~al.~\cite{did:kok:zhang}) show that in the finite
sample case $K=2,3,4$ provides best results. The reason why choosing small $K$ is often
favorable is due to a bias variance
trade off. Note that the eigenvalues occur reciprocal in the estimator $\hat\beta$
and thus larger $K$ accounts for larger instability if the eigenvalues
are close to zero. A practical approach is to chose $K$ the largest integer for
which $\hat{\lambda}_K/\hat \lambda_1\geq \gamma$, where $\gamma$ is some threshold.

\begin{Theorem}\label{th:consist}
Fix some $K\geq 1$.
Assume that the $K+1$ largest eigenvalues of the covariance operator
$C$ of $Z_k$ satisfy $\lambda_1>\lambda_2>\ldots>\lambda_{K+1}>0$.
Let $\beta(K)$ and $\hat\beta(K)$ be the operators belonging
to the kernel functions $\beta(t,s;K)$ and $\hat \beta(t,s;K)$, respectively. Let
Assumption~\ref{a:estimation} hold with condition (b) strengthened to $\alpha=4$.
Then we have
$$
\|\beta(K)-\hat\beta(K)\|_\mathcal{S}=O_P\left(N^{-1/2}\right)\quad \text{as $N\to\infty$.}
$$
\end{Theorem}
In Theorem \eqref{th:consist} $N$ obviously denotes the sample size
which is suppressed in the notation.
The proof of the theorem is given in Section~\ref{s:proofs}. Our conditions
imply that $E\|Z_k\|^4<\infty$. This assumption is probably more
stringent than necessary and a relaxation would be desirable. Note however,
that finite 4th moments are required in \cite{bosq:2000} even for i.i.d.\ $\{\nu_k\}$.

\subsection{Simulation study}\label{ss:sim}

In this section we demonstrate the capabilities of our estimators for $\beta(t,s)$ and $\delta(t)$ on simulated data.  We proceed as follows:  We will choose a simple $\beta(t,s)$ and $\delta(t)$, simulate several days of observations using these parameters, and then use the estimation procedure given in Section \ref{ss:far} to obtain $\hat{\beta}(t,s;2)$ and $\hat{\delta}(t;2)$ from \eqref{e:betahat} and \eqref{e:hat_delta} respectively.\\

We will use $\beta(t,s)=16s(1-s)t(1-t)$ and $\delta(t)=0.01$ for our simulations.  Now that we have chosen $\beta(t,s)$ and $\delta(t)$ we can simulate data according to \eqref{e:y} and \eqref{e:sigma}.  We will use $\varepsilon_i(t) = B_i(t) + N_i\sqrt{1-t(1-t)}$ for the error term, where $B_i(t)$ are iid standard Brownian bridges and $N_i$ are iid standard normals.  Note that this gives $E(\varepsilon^2(t))=1$ for all $t$, which is assumed by our estimation procedure.  After simulating $N$ days of data we compute $\hat{\beta}(t,s;2)$ and $\hat{\delta}(t;2)$.  Figures \ref{f:N30}, \ref{f:N300}, and \ref{f:N3000} show the estimates when $N=30$, $N=300$, and $N=3000$, respectively.  We see from these plots that the estimators described in Section \ref{ss:far} accurately estimate the parameters, $\beta(t,s)$ and $\delta(t)$, when the sample size is sufficiently large.  Note that each plot of $\hat{\delta}(t;2)$ has the true $\delta(t)$ superimposed.  A plot of the true $\beta(t,s)$ is given in figure \ref{f:truebeta}.

\newpage

\begin{figure}[h!]
\includegraphics[width=7cm]{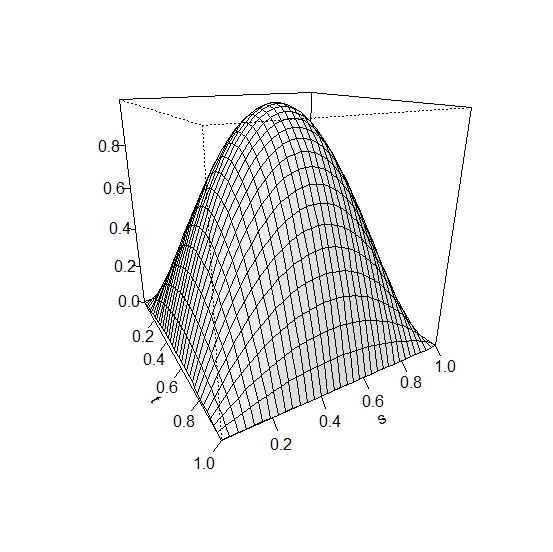}\hfill
\caption{$\beta(t,s)=16s(1-s)t(1-t)$}
\label{f:truebeta}
\vspace{1in}
\includegraphics[width=7cm]{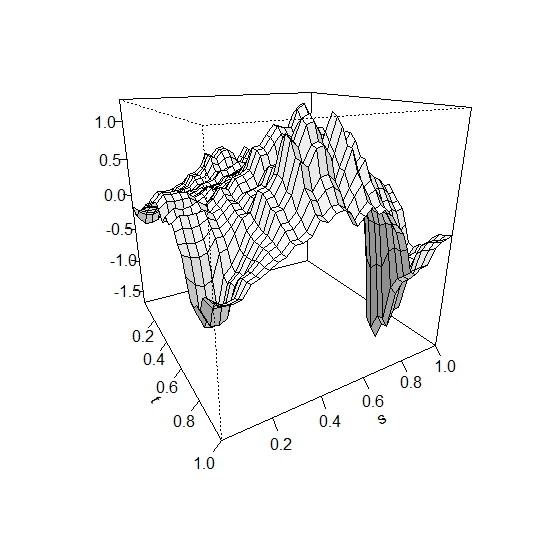}\hfill
\includegraphics[width=7cm]{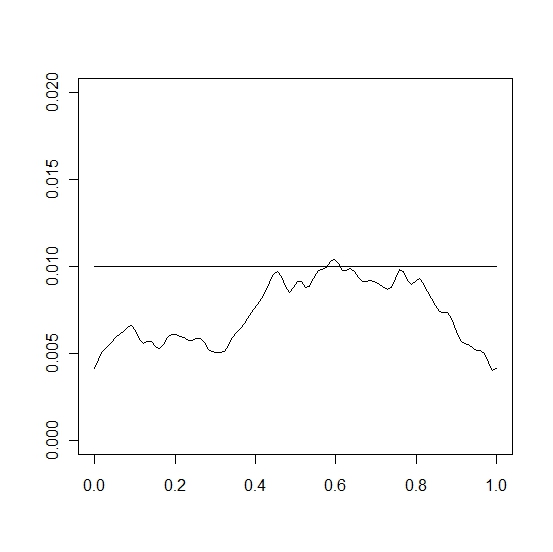}\hfill
\caption{Using a sample of size $N=30$, we obtain $\hat{\beta}(t,s;2)$ on the left and $\hat{\delta}(t;2)$ with $\delta(t)=.01$ superimposed on the right.}
\label{f:N30}

\end{figure}
\newpage

\begin{figure}[h!]

\includegraphics[width=7cm]{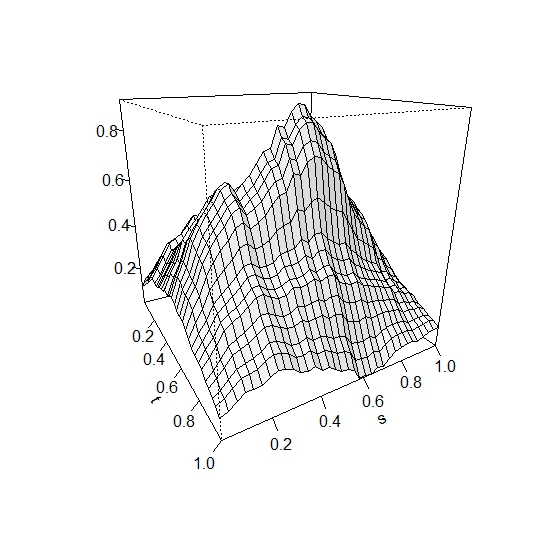}\hfill
\includegraphics[width=7cm]{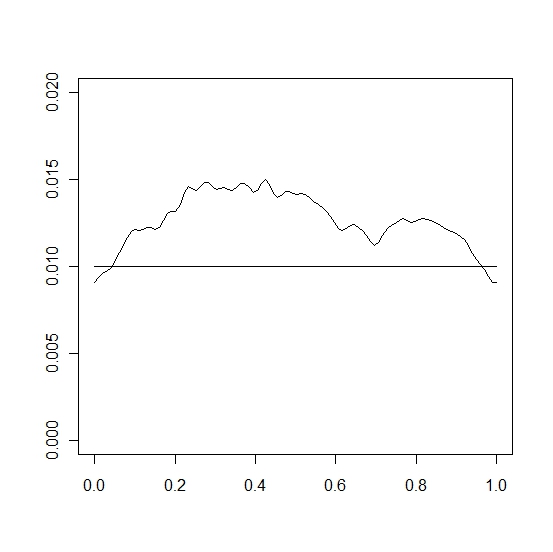}\hfill
\caption{Using a sample of size $N=300$, we obtain $\hat{\beta}(t,s;2)$ on the left and $\hat{\delta}(t;2)$ with $\delta(t)=.01$ superimposed on the right.}
\label{f:N300}
\vspace{1in}
\includegraphics[width=7cm]{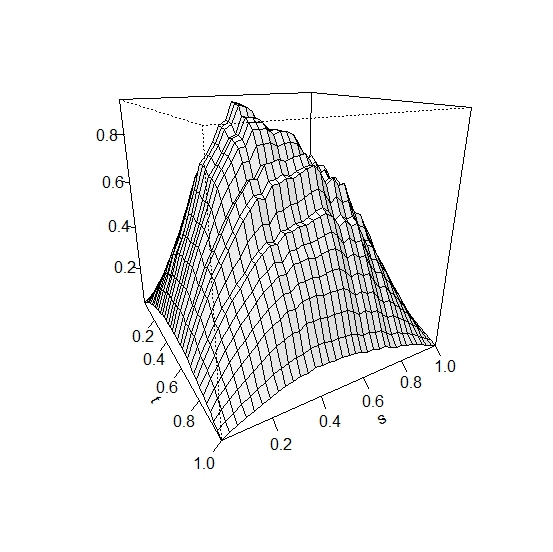}\hfill
\includegraphics[width=7cm]{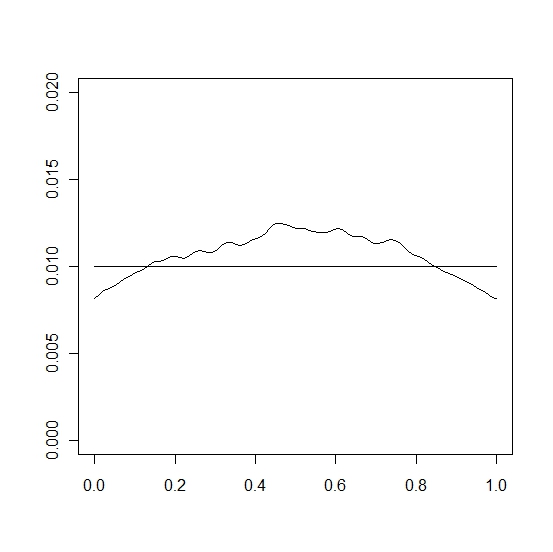}\hfill
\caption{Using a sample of size $N=3000$, we obtain $\hat{\beta}(t,s;2)$ on the left and $\hat{\delta}(t;2)$ with $\delta(t)=.01$ superimposed on the right.}
\label{f:N3000}
\end{figure}
\newpage

\section{An example}\label{s:empirical}

In this section we show an example illustrating that our model captures the basic features of intraday returns. Let $P_k(t)$ denote the price of a stock on day $k$ at time $t$.  Then $y_k(t)$ can be viewed as the log-returns of the stock, $y_k(t)=\log P_k(t) - \log P_k(t-h)$, during period $h$ (cf. Cyree et al.~\cite{cyree:2005}), where $h$ is typically 1, 5, or 15 minutes.  We will use $h=5$ for $5$-minute returns.  The volatility of the stock is then represented by $\sigma_k^2(t)=\mathrm{Var}(y_k(t)|\mathcal{F}_{k-1})$.\\

The first step to simulating the intraday returns is to estimate the parameters, $\delta(t)$ and $\beta(t,s)$, as outlined in Section \ref{ss:far}.  These parameters were estimated for the S\&P 100 index based on data from April 1, 1997 to March 30, 2007.  The estimated functions, $\hat{\beta}(t,s;2)$ and $\hat{\delta}(t;2)$, are shown in Figure \ref{fig:deltabeta}.\\

\begin{figure}[h!]
\includegraphics[width=7cm]{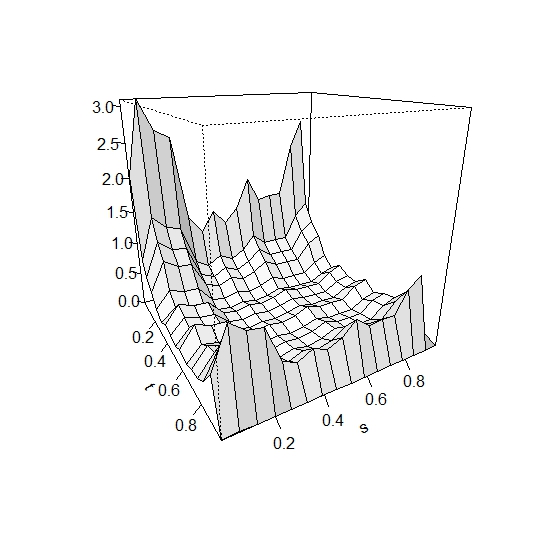}\hfill
\includegraphics[width=7cm]{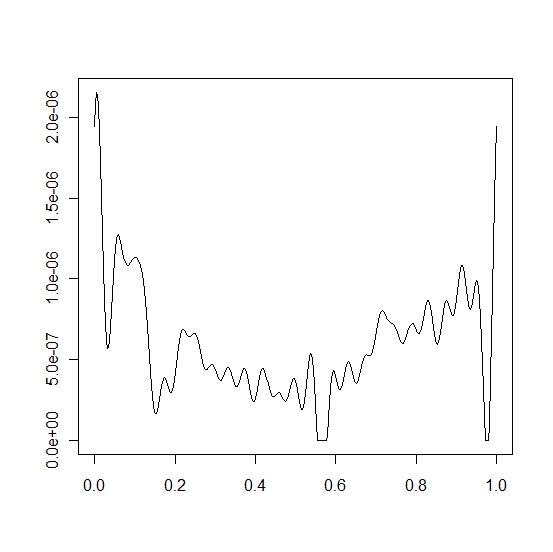}
\caption{Left: $\hat{\beta}(t,s;2)$ estimated from S\&P 100 index.  Right: $\hat{\delta}(t;2)$ estimated from S\&P 100 index.}
\label{fig:deltabeta}
\end{figure}

Notice in Figure \ref{fig:deltabeta} that $\hat{\beta}(t,s;2)$ and $\hat{\delta}(t;2)$ are somewhat larger when $t$ is close to $0$ or $1$.  According to \eqref{e:sigma} this suggests that the volatility, $\sigma_k^2(t)$, tends to be larger at the beginning and end of each trading day.  Higher volatilities at the beginning and the end of the trading day have been observed by several authors (cf.\ Gau~\cite{gau:2005} and Evans and Speight~\cite{evans:2010}).  This phenomenon is consistent with our observed log-return data based on the S\&P 100 index and is captured by our model.

Having estimated the parameters, $\delta(t)$ and $\beta(t,s)$, we can now simulate several days of observations according to \eqref{e:y} and \eqref{e:sigma}.  We will use $\varepsilon_i(t) = 2^{-200t}\sqrt{\log(2)}W_i(2^{400t}/\log(2))$ for the error term, where $W_i(t)$ are iid standard Brownian motions.  Note that this gives $E(\varepsilon^2(t))=1$ for all $t$, which is assumed by our estimation procedure.\\
\newpage

\begin{figure}[h!]
\includegraphics[width=7cm]{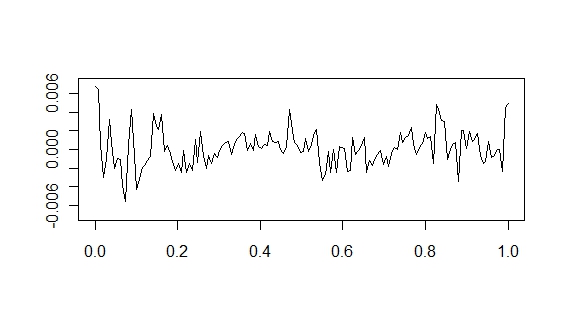}\hfill
\includegraphics[width=7cm]{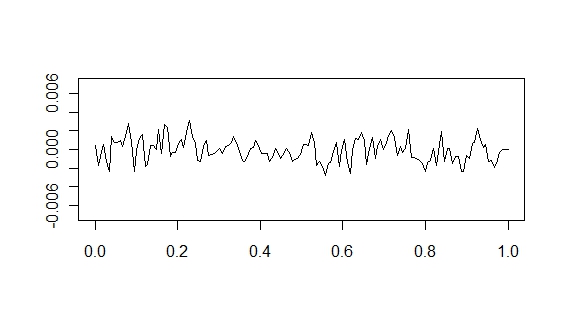}\hfill
\includegraphics[width=7cm]{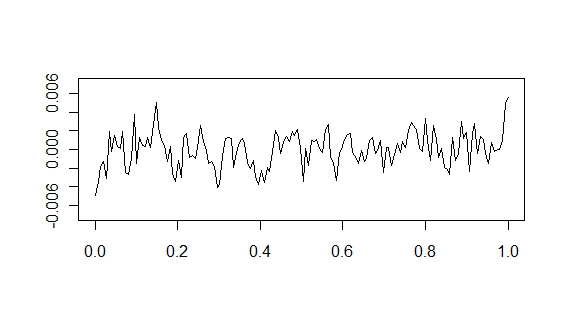}\hfill
\includegraphics[width=7cm]{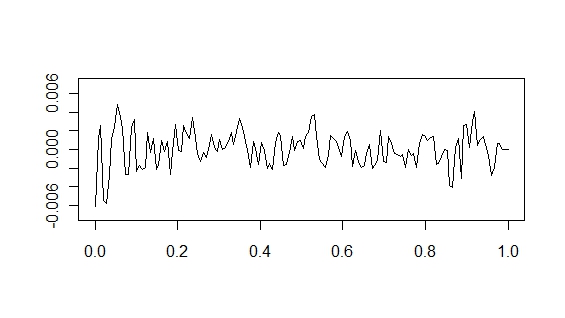}\hfill
\includegraphics[width=7cm]{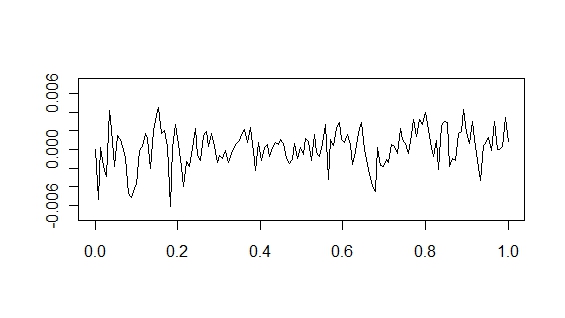}\hfill
\includegraphics[width=7cm]{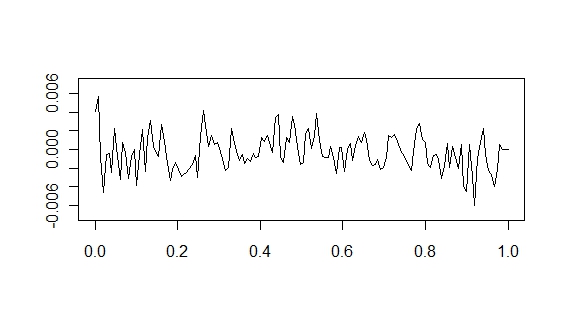}\hfill
\includegraphics[width=7cm]{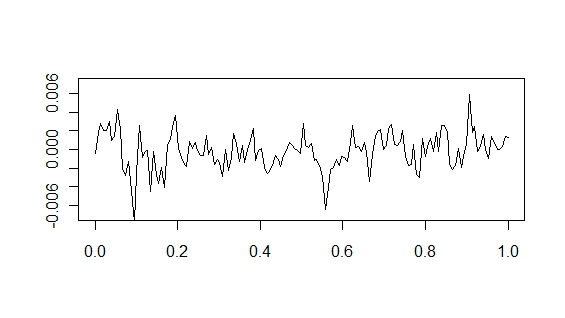}\hfill
\includegraphics[width=7cm]{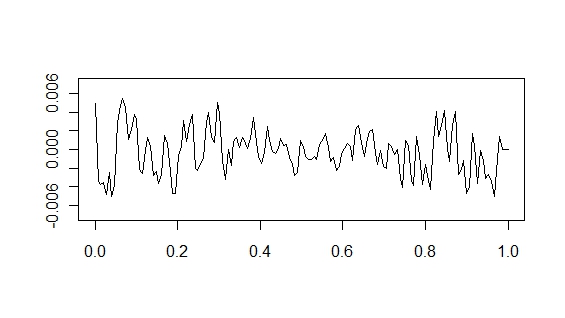}\hfill
\includegraphics[width=7cm]{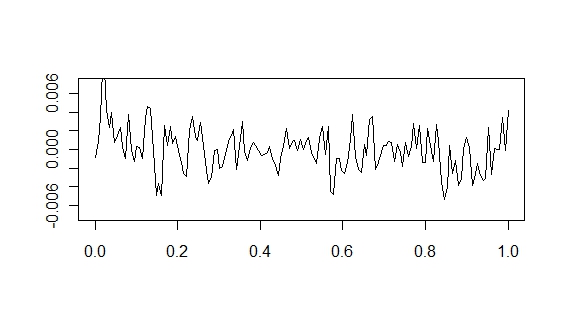}\hfill
\includegraphics[width=7cm]{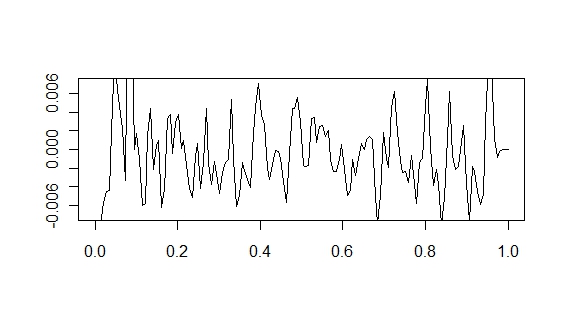}\hfill
\caption{Left panel: Five consecutive days of simulated values for $y_k(t)$ .  Right panel: 5-minute log-returns for the S\&P index between April 11 and April 15, 2000.}
\label{fig:comparison}
\end{figure}

We simulated 5 days of log-returns which we compare with the log-returns of the S\&P 100 index.  The right side of Figure \ref{fig:comparison} is the plot of the 5-minute returns on the S\&P 100 index between April 11 and April 15, 2000. The left side of  Figure \ref{fig:comparison} shows five consecutive days of simulated values for $y_k(t)$.  The simulations show that our model empirically captures the main characteristics of financial data.\\

\section{Proofs}\label{s:proofs}

The proofs of Theorems~\ref{th:Hexist} and \ref{th:Cexist} are based on general results for iterated random functions as those in Wu and Shao~\cite{wu:shoa:2004}
and Diaconis and Freedman~\cite{diaconis:freedman:1999}. For the convenience of the reader
we shall repeat here the main ideas of \cite{wu:shoa:2004}.

Let $(S,\rho)$ be a complete, separable metric space. Let $\Theta$ be another
metric space and let $M:\Theta\times S\to S$ be a measurable function.
For a random element $\theta$ with values in $\Theta$, an
iterated random function system is defined via the random
mappings $M_{\theta}(\cdot)$. More precisely it is assumed that
\begin{equation}\label{e:irfs}
X_n=M_{\theta_n}(X_{n-1}),\quad n\in \mathbb{N},
\end{equation}
where $\{\theta_n,\ -\infty < n < \infty \}$ is an i.i.d.\ sequence with values in $\Theta$.
Thereby it
is assumed that $X_0$ is independent of $\{\theta_n,\, n\geq 1\}$.
For any $x\in S$ we define
$$
S_n(x)=M_{\theta_{n}}\circ M_{\theta_{n-1}}\circ\cdots\circ M_{\theta_1}(x),\quad n=1,2,\ldots,
$$
where $\circ$ denotes the composition of functions. We also introduce the
backward version of $S_n$, which is given by
$$
Z_n(x)=M_{\theta_{-1}}\circ M_{\theta_{-2}}\circ\cdots\circ M_{\theta_{-n}}(x),\quad x\in S,\> n=1,2,\ldots.
$$
The following theorem is a slight modification of Theorem~2 of \cite{wu:shoa:2004}, so that
it is immediately applicable for our purposes.
\begin{Theorem}(Wu and Shao, 2004.)\label{th:wu}
Assume that\\[1ex]
{\em (A)}
there are $y_0\in S$ and $\alpha>0$ such that
$
E\big\{\rho(y_0,M_{\theta_0}(y_0))\big\}^\alpha<\infty
$
and\\[1ex]
{\em (B)} there are $x_0\in S$, $\alpha>0$, $0<r_1=r_1(\alpha)<1$ and
$c=c(\alpha)<\infty$ such that
$$
E\big\{\rho(S_n(x),S_n(x_0))\big\}^\alpha\leq cr_1^n\big\{\rho(x,x_0)\big\}^\alpha
$$
for all $x\in S$ and $n\in \mathbb{N}$.
Then for all $x\in S$ we have $Z_n(x)$ converges almost surely
to some $Z_\infty$ which is independent of $x$. Furthermore
$Z_\infty=g(\theta_0,\theta_{-1},\ldots)$ and
$$
E\big\{\rho(Z_n(x),Z_\infty)\big\}^\alpha\leq c_1r^n,
$$
where $c_1=c_1(x,x_0,y_0,\alpha)<\infty$ and
$0<r=r(\alpha)<1$. Moreover, the process $X_n=g(\theta_n,\theta_{n-1},\ldots)$
is a stationary solution of \eqref{e:irfs}.
Finally, if we let $X_0^*=f(\theta_0',\theta_{-1}',\ldots)$ where
$\{\theta_n'\}$ is an independent copy of
$\{\theta_n\}$, then
$$
E\big\{\rho(S_n(X_0^*),S_n(X_0))\big\}^\alpha\leq c_2r_2^n,
$$
with some $0<r_2=r_2(\alpha)<1$ and $c_2=c_2(\alpha)>0$.
\end{Theorem}
\begin{proof}[Proof of Theorem~\ref{th:Hexist}.]
We need to show that the conditions of Theorem~\ref{th:wu} are
satisfied when the underlying space is $\mathcal{H}$ with metric $\|\cdot\|_\mathcal{H}$ and
$$
M_{\theta_n}(x)(t)=\delta(t)+\int\beta(t,s)\varepsilon_{n-1}^2(s)x(s)ds.
$$
To demonstrate {\em (A)} of Theorem~\ref{th:wu} we use $y_0(t)=0$, $0\leq t\leq 1$,
and get
$$
\big\|y_0-M_{\theta_0^2}(y_0)\big\|_\mathcal{H}^2=\int\delta^2(t)dt<\infty,
$$
by assumption. Since for any $x,x_0\in \mathcal{H}$ we have
\begin{align*}
\big\|S_n(x)-S_n(x_0)\big\|_\mathcal{H}&=\big\|M_{\theta_n^2}(S_{n-1}(x))-
M_{\theta_n^2}(S_{n-1}(x_0))\big\|_\mathcal{H}\\
&=\bigg(\int\bigg(\int\beta(t,s)\big\{S_{n-1}(x)(s)-S_{n-1}(x_0)(s)\big\}
\varepsilon_{n-1}^2(s)ds\bigg)^2dt\bigg)^{1/2}\\
&\leq \bigg(\int\bigg\{\int\beta^2(t,s)\varepsilon_{n-1}^4(s)ds\bigg\}
\int\bigg\{S_{n-1}(x)(s)-S_{n-1}(x_0)(s)\bigg\}^2
dsdt\bigg)^{1/2}\\
&=K(\varepsilon_{n-1}^2)\,\big\|S_{n-1}(x)-S_{n-1}(x_0)\big\|_\mathcal{H},
\end{align*}
by the Cauchy-Schwarz inequality. Repeating the arguments above, we conclude
$$
\big\|S_n(x)-S_n(x_0)\big\|_\mathcal{H}\leq \big\|x-x_0\big\|_\mathcal{H}\prod_{i=0}^{n-1}K(\varepsilon_i^2).
$$
Taking expectations on both sides and using the independence of the $\varepsilon_i$
proves {\em (B)}.
\end{proof}

Theorem~\ref{th:Hweak} is a simple corollary to Theorem~\ref{th:Hexist} and
Theorem~\ref{th:wu}. Theorem~\ref{th:Cexist} can be proven along the same lines
of argumentation and the proof is omitted.

\begin{proof}[Proof of Propositions~\ref{p:momentsH} and \ref{p:momentsC}.]
First we establish \eqref{e:momH1}. We follow the proof of Theorem~\ref{th:Hexist}.
Since $E\big\{\|\sigma_0^2\|_\mathcal{H}\big\}^\alpha=E\big\{\|Z_\infty\|_\mathcal{H}\big\}^\alpha$, according
to the construction in the proof of Theorem~\ref{th:Hexist} we have
\begin{align*}
E\big\{\|\sigma_0^2\|_\mathcal{H}\big\}^\alpha & =  E\big\{\|Z_\infty\|_\mathcal{H}\big\}^\alpha\\
& \leq  E\big\{\|Z_1(0)\|_\mathcal{H}+\|Z_1(0)-Z_\infty(0)\|_\mathcal{H}\big\}^\alpha\\
&\leq 2^\alpha\Big\{ E\big\{\|Z_1(0)\|_\mathcal{H}\big\}^\alpha+E\big\{\|Z_1(0)-Z_\infty(0)\|_\mathcal{H}\big\}^\alpha\Big\},
\end{align*}
where $0$ denotes the "zero function" on $[0,1]$. According the
proof of Theorem~\ref{th:Hexist} and Theorem~\ref{th:wu}
the term $E\big\{\|Z_1(0)-Z_{\infty}(0)\|_\mathcal{H} \big\}^\alpha<\infty$. Furthermore,
the term $E\big\{\|Z_1(0)\|_\mathcal{H}\big\}^\alpha=\Big(\int\delta^2(t)dt\Big)^{\frac{\alpha}{2}}<\infty$.
To show \eqref{e:momH3}, we note that
\begin{align*}
E\big\{\|y_0\|_\mathcal{H}\big\}^\alpha & =  E\bigg[\int y_0^2(t)dt\bigg]^{\frac{\alpha}{2}}\\
& =  E\bigg[\int \varepsilon_0^2(t)\sigma_0^2(t)dt\bigg]^{\frac{\alpha}{2}}\\
&\leq E\big\{\|\varepsilon_0\|_\infty\big\}^\alpha E\big\{\|\sigma_0\|_\mathcal{H}\big\}^{\alpha},
\end{align*}
since $\varepsilon_0$ and $\sigma_0$ are independent processes. Proposition~\ref{p:momentsH}
is proven.

The proof of Proposition~\ref{p:momentsC} only requires minor modifications and is therefore
omitted.
\end{proof}

\begin{proof}[Proof of Proposition~\ref{p:continuity}.] Using recursion \ref{e:y}
we have
\begin{align*}
\omega(y_0,h)&=\sup_{0\leq t\leq 1-h}\sup_{0\leq s\leq h}|\varepsilon_0(t+s)\sigma_0(t+s)-
\varepsilon_0(t)\sigma_0(t)|\\
&\leq\sup_{0\leq t\leq 1-h}\sup_{0\leq s\leq h}\big\{|\varepsilon_0(t+s)||\sigma_0(t+s)-
\sigma_0(t)|+|\sigma_0(t)||\varepsilon_0(t+s)-
\varepsilon_0(t)|\big\}\\
&\leq \|\varepsilon_0\|_\infty\omega(\sigma_0,h)+\|\sigma_0\|_\infty\omega(\varepsilon_0,h).
\end{align*}
The independence of $\varepsilon_0$ and $\sigma_0$ yields
$$
E\big\{\|\sigma_0\|_\infty\omega(\varepsilon_0,h)\big\}^p=E\big\{\|\sigma_0\|_\infty\big\}^p
E\big\{\omega(\varepsilon_0,h)\big\}^p.
$$
Proposition~\ref{p:momentsC} gives $E\big\{\|\sigma_0^2\|_\infty\big\}^p<\infty$.  This implies that $E\big\{\|\sigma_0\|_\infty\big\}^p<\infty$
and therefore
$$
\lim_{h\to 0}E\big\{\|\sigma_0\|_\infty\omega(\varepsilon_0,h)\big\}^p=0.
$$
The identity $|\sqrt{a}-\sqrt{b}|\leq \sqrt{|a-b|}$, $a,b\geq 0$, implies
\begin{align*}
\omega^p(\sigma_0,h)&=\sup_{0\leq t\leq 1-h}\sup_{0\leq s\leq h}|\sigma_0(t+s)-\sigma_0(t)|^p\\
&\leq\left(\sup_{0\leq t\leq 1-h}\sup_{0\leq s\leq h}|\sigma_0^2(t+s)-\sigma_0^2(t)|\right)^{\frac{p}{2}}.
\end{align*}
Recursion \eqref{e:sigma} gives
\begin{align*}
|\sigma_0^2(t+s)-\sigma_0^2(t)|&\leq |\delta(t+s)-\delta(t)|+
\bigg|\int(\beta(t+s,r)-\beta(t,r))y_{-1}^2(r)dr\bigg|\\
&\leq \omega(\delta,h)+
\sup_{0\leq t\leq 1-h}\sup_{0\leq s\leq h}\sup_{0\leq r\leq 1}|
\beta(t+s,r)-\beta(t,r)|\times\int y_{-1}^2(r)dr.
\end{align*}
Hence
\begin{align*}
E\left(\omega^p(\sigma_0,h)\right)&= E\big\{\sup_{0\leq t\leq 1-h}\sup_{0\leq s\leq h}|\sigma_0^2(t+s)-\sigma_0^2(t)|\big\}^{\frac{p}{2}}\\
&\leq 2^{\frac{p}{2}}\Big\{
\big[
\sup_{0\leq t\leq 1-h}\sup_{0\leq s\leq h}\sup_{0\leq r\leq 1}
|\beta(t+s,r)-\beta(t,r)|
\big]^{\frac{p}{2}}\times E\big\{\|y_0\|_\infty\}^p+\big[\omega(\delta,h)\big]^{\frac{p}{2}}\Big\}.
\end{align*}
Proposition~\ref{p:momentsC} yields that $E\big\{\|y_0\|_\infty\big\}^p<\infty$
and $E\big\{\|\sigma_0\|_\infty\big\}^p<\infty$. So by the independence of the processes
$\varepsilon_0$ and $\sigma_0$ we conclude
$$
\lim_{h\to 0}\big\{\|\varepsilon_0\|_\infty\omega(\sigma_0,h)\big\}^p=0,
$$
completing the proof of Proposition~\ref{p:continuity}.
\end{proof}

\begin{proof}[Proof of Theorem~\ref{th:consist}]
Under our assumptions it follows from Theorem~\ref{th:Hweak} that for any $m\geq 1$
$$
E\|Z_k-Z_{km}\|_\mathcal{H}^4\leq \mathrm{const}\times r^m
$$
where $r\in (0,1)$ and $Z_{km}$ are the $m$--dependent approximations
of $Z_k$ (constructed by using $\sigma_{km}^2$ instead of $\sigma_k^2$ in
the definition of $Z_k$). This shows that the notion of {\em $L^4$--$m$--approximability} suggested in H\"ormann and Kokoszka~\cite{hoermann:kokoszka:2010}
applies to the sequence $\{Z_{k}\}$. As consequence we have with
$\hat c_i=\mathrm{sign}\langle \hat{e}_i,e_i\rangle$ that
\begin{align*}
(a)&\quad \max_{1\leq i\leq K}E\|\hat{c}_i\hat{e}_i-e_i\|_\mathcal{H}^2=O\left(N^{-1}\right);\\
(b)&\quad \max_{1\leq i\leq K}E|\hat\lambda_i-\lambda_i|^2=O\left(N^{-1}\right)
\end{align*}
and therefore that
\begin{align*}
(a')&\quad \max_{1\leq i\leq K}\|\hat{c}_i\hat{e}_i-e_i\|_\mathcal{H}=O_P\left(N^{-1/2}\right);\\
(b')&\quad \max_{1\leq i\leq K}|\hat\lambda_i-\lambda_i|=O_P\left(N^{-1/2}\right)
\end{align*}
(See \cite[Theorem~3.2]{hoermann:kokoszka:2010}.)
The random sign $\hat c_i$ (which we cannot observe) accounts for the
fact that $e_i$ can be only uniquely identified up to its sign. As our estimator
$\hat\beta(K)$ doesn't depend on the signs of the $\hat e_i$, this poses no
problem. We define
$$
\sigma_{i,j}=E\langle Z_1,e_i\rangle\langle Z_2,e_j\rangle
$$
and let
$$
\hat{\sigma}_{i,j}=\frac{1}{N-1}\sum_{k=1}^{N-1}\langle Z_k,\hat{e}_i\rangle\langle Z_{k+1}, 
\hat{e}_j\rangle
$$
be the empirical counterpart. Then we have
\begin{align*}
E|\sigma_{i,j}-\hat c_i\hat c_j\hat{\sigma}_{i,j}|&\leq E\left|\frac{1}{N-1}\sum_{k=1}^{N-1}
\left(\langle Z_k,{e}_i\rangle\langle Z_{k+1},{e}_j\rangle-
E\langle Z_k,{e}_i\rangle\langle Z_{k+1},{e}_j\rangle\right)\right|\\
&\quad+E\left|
\frac{1}{N-1}\sum_{k=1}^{N-1}\left(\langle Z_k,\hat c_i\hat{e}_i\rangle\langle Z_{k+1},\hat c_j\hat{e}_j\rangle
-\langle Z_k,{e}_i\rangle\langle Z_{k+1},{e}_j\rangle
\right)
\right|\\
&\quad=:T_1(i,j;N)+T_2(i,j;N).
\end{align*}
The processes
$\mathcal{Z}_k=\mathcal{Z}_k(i,j)=\langle Z_k,{e}_i\rangle\langle Z_{k+1},{e}_j\rangle$
are strictly stationary for every choice of $i$ and $j$ and we can
again define the approximations $\mathcal{Z}_{km}$ in the spirit of
Section~\ref{s:farch}. We have by independence of $\mathcal{Z}_0$ and
$\mathcal{Z}_{kk}$
\begin{align*}
\sum_{h\geq 0}|\mathrm{Cov}(\mathcal{Z}_0,\mathcal{Z}_h)|&\leq
E\mathcal{Z}_0^2+\left(E\mathcal{Z}_0^2\right)^{1/2}\times
\sum_{h\geq 1}\left(E(\mathcal{Z}_h-\mathcal{Z}_{hh})^2\right)^{1/2}\\
&\leq E\|Z_0\|_\mathcal{H}^4+E\|Z_0\|_\mathcal{H}^2\times
\sum_{h\geq 1}\left(E(\mathcal{Z}_h-\mathcal{Z}_{hh})^2\right)^{1/2}.
\end{align*}
Further we have by repeated application of the Cauchy-Schwarz
inequality that
\begin{align*}
E(\mathcal{Z}_h-\mathcal{Z}_{hh})^2&=
E\left[\langle Z_h,e_i\rangle\langle Z_{h+1},e_j\rangle-
\langle Z_{hh},e_i\rangle\langle Z_{h+1,h},e_j\rangle\right]^2\\
&\quad \leq
2\left\{
E\left[
\langle Z_h-Z_{hh},e_i\rangle\langle Z_{h+1},e_j\rangle
\right]^2+
E\left[\langle Z_{hh},e_i\rangle\langle Z_{h+1}-Z_{h+1,h},e_j\rangle\right]^2
\right\}\\
&\quad \leq
2\left\{
E\langle Z_h-Z_{hh},e_i\rangle^2E\langle Z_h,e_j\rangle^2+
E\langle Z_{hh},e_i\rangle^2 E\langle Z_{h+1}-Z_{h+1,h},e_j\rangle^2
\right\}\\
&\quad\leq 2E\|Z_{0}\|^2_\mathcal{H} \left\{
E\| Z_h-Z_{hh}\|_\mathcal{H}^2 +
E\| Z_{h+1}-Z_{h+1,h}\|_\mathcal{H}^2
\right\}\\
&\quad\leq  \mathrm{const}\times r^h,
\end{align*}
for some $r\in (0,1)$. This proves that the autocovariances of
the process $\{\mathcal{Z}_k\}$ are absolutely summable.
A well known result in time series analysis thus implies that
$$
(N-1)\mathrm{Var}\left(\frac{1}{N-1}\sum_{k=1}^{N-1}\mathcal{Z}_k\right)\leq
2\sum_{h\geq 0}|\mathrm{Cov}(\mathcal{Z}_0,\mathcal{Z}_h)|\leq c_0,\quad\forall N\geq 2,
$$
(see e.g.\ the proof of Theorem~7.1.1.\ in Brockwell and Davis~\cite{brockwell:davis:1991})
where, as we have shown,
the constant $c_0$ is independent of the choice of $i$ and $j$ in
the definition of $\mathcal{Z}_k$. Hence $\max_{1\leq i,j\leq K}T_1(i,j;N)=O_P\left(N^{-1/2}\right)$.

Using relation ($a$) above, one can show that also
$\max_{1\leq i,j\leq K}T_2(i,j;N)=O_P\left(N^{-1/2}\right)$.
We thus have
\begin{align*}
(c)\quad \max_{1\leq i,j\leq K}|\sigma_{i,j}-\hat c_i\hat c_j\hat{\sigma}_{i,j}|=O_P\left(N^{-1/2}\right).
\end{align*}

We have now the necessary tools to prove Theorem~\ref{th:consist}. By
relations ($a'$), ($b'$) and ($c$) we have that
\begin{align*}
&\left\|\beta(K)-\hat\beta(K)\right\|_\mathcal{S}=
\left\|\sum_{1\leq i,j\leq K}\left(
\frac{\sigma_{j,i}}{\lambda_j} e_j\otimes e_i-
\frac{\hat\sigma_{j,i}}{\hat\lambda_j} \hat e_j\otimes \hat e_i
\right)\right\|_\mathcal{S}\\
&\quad\leq\sum_{1\leq i,j\leq K}\left\{\left|
\frac{\sigma_{j,i}}{\lambda_j}-
\frac{\hat c_j\hat c_i\hat\sigma_{j,i}}{\hat\lambda_j}
\right|+\left|
\frac{\sigma_{j,i}}{\hat\lambda_j}\right|\left\| e_j\otimes e_i-
 \hat c_j\hat e_j\otimes \hat c_i\hat e_i
\right\|_\mathcal{S}\right\}\\
&\quad\leq K^2\left\{O_P\left(N^{-1/2}\right)+O_P(1)\max_{1\leq i,j\leq K}\left\| e_j\otimes e_i-
 \hat c_j\hat e_j\otimes \hat c_i\hat e_i
\right\|_\mathcal{S}\right\}.
\end{align*}
The proof follows from
$
\left\| e_j\otimes e_i-
\hat c_j \hat e_j\otimes \hat c_i\hat e_i
\right\|_\mathcal{S}\leq \|\hat c_j\hat e_j- e_j\|_\mathcal{H}+
\|\hat c_i\hat e_i- e_i\|_\mathcal{H}.
$
\end{proof}

\end{document}